\documentclass[11pt,a4paper]{article}

\usepackage{tikz}
\usepackage{amsthm}
\usepackage{amsmath}
\usepackage{amssymb}
\usepackage{mathrsfs}
\usepackage{geometry}
\usepackage{graphicx}
\usepackage{amsfonts}
\usepackage{epstopdf}
\usepackage{enumerate}
\usepackage{hyperref}
\usepackage{subcaption}

\allowdisplaybreaks

\newcommand{\dd}{\partial}

\newcommand{\Rmnum}[1]{\uppercase\expandafter{\romannumeral#1}} 

\newcommand{\bbE}{\mathbb{E}}
\newcommand{\bbP}{\mathbb{P}}

\newcommand{\calC}{\mathcal{C}}

\newcommand{\myset}[1]{\left\{#1\right\}}

\newtheorem{mythm}{Theorem}[section]
\newtheorem{myprop}[mythm]{Proposition}
\newtheorem{mylem}[mythm]{Lemma}
\newtheorem{mycor}[mythm]{Corollary}
\newtheorem{myrmk}[mythm]{Remark}

\AtEndDocument{
\bigskip{\footnotesize%
  \textsc{Fakult\"{a}t f\"{u}r Mathematik, Universit\"{a}t Bielefeld, Postfach 100131, 33501 Bielefeld, Germany.} \par
  \textsc{Universit\'e Grenoble Alpes, CNRS UMR 5582, Institut Fourier, Gi\`eres, France.} \par
  \textit{E-mail address}: \texttt{yangmengqh@gmail.com}
}}

\begin{document}

\title{{\Large{Metrics and Uniform Harnack Inequality on the Strichartz Hexacarpet}}}
\author{Meng Yang}
\date{}

\maketitle

\abstract{We construct \emph{intrinsic} metrics on the Strichartz hexacarpet using weight functions and show that these metrics do \emph{not} satisfy the chain condition. We give uniform Harnack inequality on the approximating graphs of the Strichartz hexacarpet with respect to the intrinsic metrics instead of graph metrics.}

\footnote{\textsl{Date}: \today}
\footnote{The author was supported by SFB1283 of the German Research Council (DFG). The author was very grateful to Prof. Alexander Grigor'yan, Prof. Alexander Teplyaev, Prof. Jun Kigami and Dr. Qingsong Gu for very helpful discussions.}

\section{Introduction}

A big open question in analysis on fractals is to construct a Brownian motion, or equivalently, a local regular Dirichlet form on any given fractal. This has been done on many fractals, for example, the Sierpi\'nski gasket (SG) \cite{BP88,Kig89} and more general post critically finite (p.c.f.) self-similar sets \cite{Kig93,Kig01,HMT06} and finitely ramified fractals \cite{Pei14}, the Sierpi\'nski carpet (SC) \cite{BB89,KZ92} and higher dimensional SCs \cite{BB99a}. Recently, Grigor'yan and the author \cite{GY19,Yan17} gave a unified purely analytic construction on the SG and the SC.

On p.c.f. self-similar sets and finitely ramified fractals, the most intrinsically essential ingredient in the construction of Brownian motion is the so-called compatible condition. However, on non-p.c.f. self-similar sets and infinitely ramified fractals, compatible condition does not hold and uniform Harnack inequality is a key ingredient which provides compactness results for appropriate approximating sequences. But uniform Harnack inequality is not easy to verify and was obtained only on the SC and higher dimensional SCs.

The main purpose of this paper is to consider another concrete non-p.c.f. self-similar set and infinitely ramified fractal, that is, the Strichartz hexacarpet. The group of Teplyaev \cite{BKNPPT12,KPBT19} has given some results on this fractal mainly on the approximating graphs, but the existence of Brownian motion still remains a conjecture. Since the Strichartz hexacarpet is defined in a very abstract way, there was not even a canonical metric, needless to say uniform Harnack inequality.

In this paper, we construct intrinsic metrics on the Strichartz hexacarpet using weight functions and give uniform Harnack inequality on the approximating graphs of the Strichartz hexacarpet with respect to the intrinsic metrics instead of graph metrics. We will see that the intrinsic metrics behave very different from graph metrics due to the unusual connectedness property of the Strichartz hexacarpet.

The construction of metrics using weight functions was initiated by Kameyama \cite{Kam00} and developed by Kigami \cite{Kig09}. Recently, Gu, Qiu and Ruan \cite{GQR18} constructed metrics on the SC using weight functions with two parameters $a$ and $b$. They showed that the weight functions give metrics if and only if $a,b\in(0,1)$ satisfy $2a+b\ge1$ and $a+2b\ge1$. They showed that the metrics satisfy the chain condition if and only if $2a+b=1$ or $a+2b=1$, that is, the point $(a,b)$ lies on part of the boundary of the admissible region to give metrics.

On the Strichartz hexacarpet, we will construct metrics using weight functions with one parameter $\mu$. We will show that the weight functions give metrics if and only if $\mu\in[1/2,1)$. However, we will show that for all $\mu\in[1/2,1)$, the metrics do not satisfy the chain condition. Hence, unlike the case on the SC, one can not obtain a metric satisfying the chain condition by adjusting the parameter on the Strichartz hexacarpet.

\section{Statement of the Main Results}\label{sec_main}

Let
$$W=\myset{0,1,2,3,4,5}.$$
Let $W_{0}=\myset{\emptyset}$ and
$$W_n=W^n=\myset{w=w_1\ldots w_n:w_i\in W,i=1,\ldots,n}\text{ for all }n\ge1.$$
Let $W_*=\cup_{n=0}^\infty W_n=\cup_{n=0}^\infty W^n$ and
$$W_\infty=W^\infty=\myset{w=w_1w_2\ldots:w_i\in W,i=1,2,\ldots}.$$

For all $n\ge0$, for all $w\in W_n$, denote
$$|w|=n.$$
We use the convention that $|\emptyset|=0$.

For all $n\ge1$, for all $w=w_1\ldots w_{n-1}w_n\in W_n$, denote
$$w^-=w_1\ldots w_{n-1}\in W_{n-1}.$$
For all $w^{(1)}=w^{(1)}_1\ldots w^{(1)}_m\in W_m$ and $w^{(2)}=w^{(2)}_1\ldots w^{(2)}_n\in W_n$, denote
$$w^{(1)}w^{(2)}=w^{(1)}_1\ldots w^{(1)}_mw^{(2)}_1\ldots w^{(2)}_n\in W_{m+n}.$$
For all $w^{(1)}=w^{(1)}_1\ldots w^{(1)}_m\in W_m$ and $w^{(2)}=w^{(2)}_1w^{(2)}_2\ldots\in W_\infty$, denote
$$w^{(1)}w^{(2)}=w^{(1)}_1\ldots w^{(1)}_mw^{(2)}_1w^{(2)}_2\ldots\in W_{\infty}.$$
For all $i\in W$, denote
\begin{align*}
i^n&=\underbrace{i\ldots i}_{n\ \text{times}}\in W_n,\\
i^\infty&=ii\ldots\in W_\infty.
\end{align*}

For all $w^{(1)}=w^{(1)}_1w^{(1)}_2\ldots,w^{(2)}=w^{(2)}_1w^{(2)}_2\ldots\in W_\infty$, define
$$s(w^{(1)},w^{(2)})=\min\myset{i\ge1:w^{(1)}_i\ne w^{(2)}_i},$$
with the convention that $\min\emptyset=+\infty$. It is obvious that
$$s(w^{(1)},w^{(2)})\ge\min{\myset{s(w^{(1)},w^{(3)}),s(w^{(3)},w^{(2)})}}\text{ for all }w^{(1)},w^{(2)},w^{(3)}\in W_\infty.$$

Fix arbitrary $r\in(0,1)$, for all $w^{(1)},w^{(2)}\in W_\infty$, let
$$\delta_r(w^{(1)},w^{(2)})=r^{s(w^{(1)},w^{(2)})},$$
with the convention that $r^{+\infty}=0$. It is obvious that for all $w^{(1)},w^{(2)},w^{(3)}\in W_\infty$, we have
$$\delta_r(w^{(1)},w^{(2)})\le\max\myset{\delta_r(w^{(1)},w^{(3)}),\delta_r(w^{(3)},w^{(2)})}.$$
Hence $\delta_r$ is an ultrametric on $W_\infty$. By \cite[Theorem 1.2.2]{Kig01}, $(W_\infty,\delta_r)$ is a compact metric space.

For all $i\in W$, define $\sigma_i:W_\infty\to W_\infty$ by
$$w=w_1w_2\ldots\mapsto \sigma_i(w)=iw_1w_2\ldots.$$

For all $w\in W_*$, for all $i\in W$, let $j=i+1(\mathrm{mod}\ 6)$, for all
$$v\in\myset{0,5}^\infty=\myset{w=w_1w_2\ldots:w_i=0,5,i=1,2,\ldots}.$$
If $i$ is even, then define
$$wi1v\sim wj1v\text{ and }wi2v\sim wj2v.$$
If $i$ is odd, then define
$$wi3v\sim wj3v\text{ and }wi4v\sim wj4v.$$

It is obvious that $\sim$ is an equivalence relation on $W_\infty$. Let $K=W_\infty\//\sim$ be equipped with the quotient topology and $\pi:W_\infty\to K$ the quotient map. Since at most two elements in $W_\infty$ are mapped to the same point in $K$, a simple topological argument gives that $K$ is a compact Hausdorff space. For all $i\in W$, for all $w^{(1)},w^{(2)}\in W_\infty$, since $w^{(1)}\sim w^{(2)}$ if and only if $\sigma_i(w^{(1)})\sim\sigma_i(w^{(2)})$, there exists a unique map $f_i:K\to K$ such that $\pi\circ\sigma_i=f_i\circ\pi$. Therefore, $K$ is a topological self-similar set, see \cite[Definition 0.3]{Kam00}. By \cite[Theorem 1.5]{Kam00}, $K$ is metrizable. $K$ is called the Strichartz hexacarpet, see \cite[FIGURE 2, FIGURE 4]{BKNPPT12} for related figures.

We use $w\in W_\infty$ also to denote the corresponding point $\pi(w)\in K$.

For all $w=w_1\ldots w_n\in W_*$, let
\begin{align*}
f_w&=f_{w_1}\circ\ldots\circ f_{w_n},\\
K_w&=f_{w_1}\circ\ldots\circ f_{w_n}(K),
\end{align*}
where $f_\emptyset=\mathrm{id}$ is the identity map. We say that $K_w$ is an $n$-cell.

We introduce a pseudo-metric given in \cite{Kam00} as follows.

Let $\mu\in(0,1)$, for all $w\in W_*$, let $g_\mu(w)=\mu^{|w|}$. 

We say that $\myset{w^{(1)},\ldots,w^{(m)}}$ is a chain if $w^{(i)}\in W_*$ for all $i=1,\dots,m$ and $K_{w^{(i)}}\cap K_{w^{(i+1)}}\ne\emptyset$ for all $i=1,\ldots,m-1$. We say that a chain $\myset{w^{(i)},\ldots,w^{(j)}}$ is a sub-chain of $\myset{w^{(1)},\ldots,w^{(m)}}$ for all $i\le j$. Denote $\calC$ as the set of all chains.

We say that $\sum_{i=1}^mg_\mu(w^{(i)})$ is the weight of the chain $\myset{w^{(1)},\ldots,w^{(m)}}$.

For all $x,y\in K$, we say that $\myset{w^{(1)},\ldots,w^{(m)}}$ is a chain connecting $x$ and $y$ if it is a chain satisfying $x\in K_{w^{(1)}}$ and $y\in K_{w^{(m)}}$. Denote $\calC(x,y)$ as the set of all chains connecting $x$ and $y$.

For all $x,y\in K$, let
$$d_\mu(x,y)=\inf\myset{\sum_{i=1}^mg_\mu(w^{(i)}):\myset{w^{(1)},\ldots,w^{(m)}}\in\calC(x,y)}.$$
Then $d_\mu$ is a pseudo-metric by the remark in \cite[Definition 1.10]{Kam00}, that is, $d_\mu(x,y)\ge0$, $d_\mu(x,y)=d_\mu(y,x)$ and $d_\mu(x,y)\le d_\mu(x,z)+d_\mu(z,y)$ for all $x,y,z\in K$.

The main results of this paper are as follows.

\begin{mythm}\label{thm_metric}
$d_\mu$ is a metric if and only if $\mu\in[1/2,1)$. For all $\mu\in[1/2,1)$, for all $i\in W$, for all $x,y\in K$, we have
$$d_\mu(f_i(x),f_i(y))=\mu d_\mu(x,y).$$
For all $w\in W_*$, we have
$$\mathrm{diam}_\mu(K_w):=\sup\myset{d_\mu(x,y):x,y\in K_w}=\mu^{|w|}.$$
The Hausdorff dimension of $(K,d_\mu)$ is $\alpha=-\log6/\log\mu$ and the normalized Hausdorff measure $\nu$ of dimension $\alpha$ exists.
\end{mythm}

\begin{myrmk}
If $\mu\in[1/2,1)$, then by \cite[Proposition 1.11]{Kam00}, $d_\mu$ is compatible with the topology of $K$. Hence $(K,d_\mu)$ is a compact metric space. For all $x\in K$, for all $r\in(0,1)$, denote
$$B_\mu(x,r)=\myset{y\in K:d_\mu(x,y)<r}.$$
\end{myrmk}

\begin{myprop}\label{prop_chain}
For all $\mu\in[1/2,1)$, $d_\mu$ does \emph{not} satisfy the chain condition, for all $\theta\in(0,-\log\mu/\log2)$, $d_\mu$ does satisfy the $\theta$-chain condition.
\end{myprop}

Let $V_0=\myset{ij^\infty:i\in W,j=0,5}$ and
$$V_{n+1}=\bigcup_{i\in W}\sigma_i(V_n)=\myset{wij^\infty:w\in W_{n+1},i\in W,j=0,5}\text{ for all }n\ge0.$$
For all $n\ge0$, let $H_n$ be the graph with vertex set $V_n$ and edge set given by
$$\myset{(w^{(1)},w^{(2)}):w^{(1)}=wv^{(1)},w^{(2)}=wv^{(2)},w\in W_n,v^{(1)},v^{(2)}\in V_0,v^{(1)}\ne v^{(2)}}.$$
There are two metrics on $H_n$, one is the usual graph metric, the other is the metric induced from the intrinsic metric $d_\mu$ on $K$.

\begin{mythm}\label{thm_Harnack}
For all $\mu\in[1/2,1)$, there exists some positive constant $C$ such that for all $x\in K$, for all $r\in(0,1)$, for all non-negative harmonic function $u$ in $V_n\cap B_\mu(x,r)$, we have
$$\max_{V_n\cap B_\mu(x,\mu r)}u\le C\min_{V_n\cap B_\mu(x,\mu r)}u.$$
\end{mythm}

\begin{myrmk}
The harmonicity is defined using graphs. The balls are defined using the intrinsic metric instead of graph metrics.
\end{myrmk}

This paper is organized as follows. In Section \ref{sec_metric}, we prove Theorem \ref{thm_metric}. In Section \ref{sec_chain}, we prove Proposition \ref{prop_chain}. In Section \ref{sec_Harnack}, we prove Theorem \ref{thm_Harnack}.

NOTATION. The letters $c,C$ will always refer to some positive constants and may change at each occurrence. The sign $\asymp$ means that the ratio of the two sides is bounded from above and below by positive constants. The sign $\lesssim$ ($\gtrsim$) means that the LHS is bounded by positive constant times the RHS from above (below).

\section{Proof of Theorem \ref{thm_metric}}\label{sec_metric}

We consider the case $\mu\in(0,1/2)$ as follows.

\begin{mylem}\label{lem_notmetric}
For all $\mu\in(0,1/2)$, we have $d_\mu(0^\infty,10^\infty)=0$, hence $d_\mu$ is not a metric.
\end{mylem}

\begin{proof}
For all $n\ge1$, we construct a chain $\myset{w^{(n,1)},\ldots,w^{(n,2^n)}}\subseteq W_n$ connecting $0^\infty$ and $10^\infty$ as follows. For $n=1$, let $w^{(1,1)}=0$ and $w^{(1,2)}=1$.

Assume that we have constructed a chain $\myset{w^{(n,1)},\ldots,w^{(n,2^n)}}\subseteq W_n$ connecting $0^\infty$ and $10^\infty$. Then for $n+1$, for all $i=1,\ldots,2^n$, let
$$w^{(n+1,i)}=0w^{(n,i)},w^{(n+1,2^n+i)}=1w^{(n,2^n+1-i)}.$$

The following facts are obvious from the above construction.
\begin{itemize}
\item For all $n\ge1$, we have $\myset{w^{(n,1)},\ldots,w^{(n,2^n)}}\subseteq W_n$.
\item For all $n\ge1$, we have $w^{(n,1)}=0^n$ and $w^{(n,2^n)}=10^{n-1}$, hence $0^\infty\in K_{w^{(n,1)}}$ and $10^\infty\in K_{w^{(n,2^n)}}$.
\item $w^{(n+1,2^n)}=010^{n-1}$ and $w^{(n+1,2^n+1)}=110^{n-1}$.
\end{itemize}

To show that $\myset{w^{(n+1,1)},\ldots,w^{(n+1,2^{n+1})}}\in\calC(0^\infty,10^\infty)$, we only need to show that $K_{w^{(n+1,2^n)}}\cap K_{w^{(n+1,2^n+1)}}\ne\emptyset$, that is, $K_{010^{n-1}}\cap K_{110^{n-1}}\ne\emptyset$.

Indeed, $010^\infty\in K_{010^{n-1}}$ and $110^\infty\in K_{110^{n-1}}$. By definition, we have $010^\infty\sim110^\infty$, that is, $010^\infty$ and $110^\infty$ are indeed the same point in $K$, hence $K_{010^{n-1}}\cap K_{110^{n-1}}\ne\emptyset$.

By induction principle, we obtain a chain $\myset{w^{(n,1)},\ldots,w^{(n,2^n)}}\subseteq W_n$ connecting $0^\infty$ and $10^\infty$ for all $n\ge1$. Hence $d_\mu(0^\infty,10^\infty)\le2^n\mu^n\to0$ as $n\to+\infty$, hence $d_\mu(0^\infty,10^\infty)=0$. Since $0^\infty$ and $10^\infty$ are distinct points in $K$, we have $d_\mu$ is not a metric.
\end{proof}

We assume that $\mu\in[1/2,1)$ hereafter. We need do some preparations as follows.

We say that a chain $\myset{w^{(1)},\ldots,w^{(m)}}$ satisfies only adjacent intersection (OAI) condition if the following conditions are satisfied.
\begin{itemize}
\item There exists no $|i-j|\ge2$ such that $K_{w^{(i)}}\cap K_{w^{(j)}}\ne\emptyset$.
\item There exists no $i\ne j$ such that $K_{w^{(i)}}\subseteq K_{w^{(j)}}$.
\end{itemize}

\begin{mylem}\label{lem_OAI}
$$d_\mu(x,y)=\inf\myset{\sum_{i=1}^mg_\mu(w^{(i)}):\myset{w^{(1)},\ldots,w^{(m)}}\in\calC(x,y)\text{ satisfies (OAI) condition}}.$$
\end{mylem}

\begin{proof}
It is obvious that the LHS $\le$ the RHS.

Assume that $\myset{w^{(1)},\ldots,w^{(m)}}\in\calC(x,y)$.

If there exist $i+2\le j$ such that $K_{w^{(i)}}\cap K_{w^{(j)}}\ne\emptyset$, then
$$\myset{w^{(1)},\ldots,w^{(i-1)},w^{(i)},w^{(j)},w^{(j+1)},\ldots,w^{(m)}}\in\calC(x,y)$$
and
$$\sum_{k=1}^{i}g_\mu(w^{(k)})+\sum_{k=j}^{m}g_\mu(w^{(k)})<\sum_{k=1}^mg_\mu(w^{(k)}).$$

If there exist $j+2\le i$ such that $K_{w^{(j)}}\cap K_{w^{(i)}}\ne\emptyset$, then
$$\myset{w^{(1)},\ldots,w^{(j-1)},w^{(j)},w^{(i)},w^{(i+1)},\ldots,w^{(m)}}\in\calC(x,y)$$
and
$$\sum_{k=1}^{j}g_\mu(w^{(k)})+\sum_{k=i}^{m}g_\mu(w^{(k)})<\sum_{k=1}^mg_\mu(w^{(k)}).$$

If there exist $i<j$ such that $K_{w^{(i)}}\subseteq K_{w^{(j)}}$, then
$$\myset{w^{(1)},\ldots,w^{(i-1)},w^{(j)},\ldots,w^{(m)}}\in\calC(x,y)$$
and
$$\sum_{k=1}^{i-1}g_\mu(w^{(k)})+\sum_{k=j}^{m}g_\mu(w^{(k)})<\sum_{k=1}^mg_\mu(w^{(k)}).$$

If there exist $j<i$ such that $K_{w^{(i)}}\subseteq K_{w^{(j)}}$, then
$$\myset{w^{(1)},\ldots,w^{(j)},w^{(i+1)},\ldots,w^{(m)}}\in\calC(x,y)$$
and
$$\sum_{k=1}^{j}g_\mu(w^{(k)})+\sum_{k=i+1}^{m}g_\mu(w^{(k)})<\sum_{k=1}^mg_\mu(w^{(k)}).$$

Repeating the above procedure finitely many times, we eventually obtain a chain still in $\calC(x,y)$ satisfying (OAI) condition with less weight than the origin chain. Hence the RHS $\le$ the LHS.

Therefore, we obtain the desired result.
\end{proof}

For all $w\in W_*$, the boundary $\dd K_w$ is given by
$$\dd K_w=\myset{wiv:i\in W,v\in\myset{0,5}^\infty},$$
the interior $\mathrm{int}(K_w)$ is given by
$$\mathrm{int}(K_w)=K_w\backslash\dd K_w.$$

We collect some basic facts as follows.

\begin{mylem}\label{lem_metric_basic}
\hspace{0em}
\begin{enumerate}[(1)]
\item For all $w\in W_*$, $\dd K_w$ is the \emph{disjoint} union of $\dd K_w\cap K_{w0}$, \ldots, $\dd K_w\cap K_{w5}$, that is,
$$\dd K_w=\coprod_{i\in W}\left(\dd K_w\cap K_{wi}\right),$$
where for all $i\in W$,
$$\dd K_w\cap K_{wi}=\myset{wiv:v\in\myset{0,5}^\infty}.$$
\item For all $n\ge1$, for all $w\in W_n$, there exist at most \emph{three} elements $v\in W_n$ with $v\ne w$ such that $K_v\cap K_w\ne\emptyset$. More precisely, there exist \emph{two} elements $v\in W_n$ with $v\ne w$ and $v^-=w^-$ such that $K_v\cap K_w\ne\emptyset$ and there exists at most \emph{one} element $v\in W_n$ with $v\ne w$ and $v^-\ne w^-$ such that $K_v\cap K_w\ne\emptyset$.
\end{enumerate}
\end{mylem}

For all $w\in W_*$, we say that $\myset{w^{(1)},\ldots,w^{(m)}}$ is a chain going through $K_w$ if it is a chain satisfying $K_{w^{(i)}}\subseteq K_w$ for all $i=1,\ldots,m$, $K_{w^{(1)}}\cap\dd K_w\ne\emptyset$ and $K_{w^{(m)}}\cap\dd K_w\ne\emptyset$. Denote $\calC(K_w)$ as the set of all chains going through $K_w$. Moreover, if there exist $j_1,j_2\in W$ with $j_1\ne j_2$ such that $\emptyset\ne\dd K_w\cap K_{w^{(1)}}\subseteq\dd K_w\cap K_{wj_1}$ and $\emptyset\ne\dd K_w\cap K_{w^{(m)}}\subseteq\dd K_w\cap K_{wj_2}$, then we say that $\myset{w^{(1)},\ldots,w^{(m)}}$ is a chain going through $K_w$ with different entries, denoted as $\myset{w^{(1)},\ldots,w^{(m)}}\in\calC(K_w)$ with different entries, it is obvious that $|w^{(i)}|\ge|w|+1$ for all $i=1,\ldots,m$.

\begin{mylem}\label{lem_metric_pre1}
For all $w\in W_*$, for all $\myset{w^{(1)},\ldots,w^{(m)}}\in\calC(K_w)$ with different entries, we have
$$\sum_{i=1}^mg_\mu(w^{(i)})\ge\mu^{|w|}.$$
\end{mylem}

\begin{proof}
Denote $n=|w|$. By the proof of Lemma \ref{lem_OAI}, we may assume that $\myset{w^{(1)},\ldots,w^{(m)}}$ satisfies (OAI) condition. Let $j_1,j_2\in W$ with $j_1\ne j_2$ satisfy $\emptyset\ne\dd K_w\cap K_{w^{(1)}}\subseteq\dd K_w\cap K_{wj_1}$ and $\emptyset\ne\dd K_w\cap K_{w^{(m)}}\subseteq\dd K_w\cap K_{wj_2}$.

Let
$$k=\max\myset{|w^{(i)}|:i=1,\ldots,m}.$$

If $k=n+1$ or $k=n+2$, then direct calculation gives the desired result. Assume that this result holds for $n+1,n+2,\ldots,k-1$. For $k>n+2$, we only need to find some $\myset{v^{(1)},\ldots,v^{(l)}}\in\calC(K_w)$ with different entries satisfying $\emptyset\ne\dd K_w\cap K_{v^{(1)}}\subseteq\dd K_w\cap K_{wj_1}$ and $\emptyset\ne\dd K_w\cap K_{v^{(l)}}\subseteq\dd K_w\cap K_{wj_2}$ and
$$\max\myset{|v^{(i)}|:i=1,\ldots,l}<k$$
such that
$$\sum_{i=1}^lg_\mu(v^{(i)})\le\sum_{i=1}^mg_\mu(w^{(i)}).$$
Then by induction assumption, we have
$$\sum_{i=1}^mg_\mu(w^{(i)})\ge\sum_{i=1}^lg_\mu(v^{(i)})\ge\mu^n.$$

If $|w^{(1)}|=k$, then $|w^{(2)}|=k$ and $(w^{(1)})^-=(w^{(2)})^-$. Indeed, suppose that $|w^{(2)}|<k$, since $K_{w^{(1)}}\not\subseteq K_{w^{(2)}}$ and $\emptyset\ne\dd K_w\cap K_{w^{(1)}}\subseteq\dd K_w\cap K_{wj_1}$, by Lemma \ref{lem_metric_basic}, we have $K_{w^{(2)}}\cap K_w\subseteq\dd K_w$, contradicting to the fact that $K_{w^{(2)}}\subseteq K_w$. Suppose that $|w^{(2)}|=k$ and $(w^{(1)})^-\ne(w^{(2)})^-$, since $\emptyset\ne\dd K_w\cap K_{w^{(1)}}\subseteq\dd K_w\cap K_{wj_1}$, by Lemma \ref{lem_metric_basic} again, we have $K_{w^{(2)}}\cap K_w\subseteq\dd K_w$, contradicting to the fact that $K_{w^{(2)}}\subseteq K_w$.

Let
$$j=\max\myset{j:|w^{(1)}|=\ldots=|w^{(j)}|,(w^{(1)})^-=\ldots=(w^{(j)})^-},$$
then $j\ge2$. Hence we have $\myset{(w^{(1)})^-,w^{(j+1)},\ldots,w^{(m)}}\in\calC(K_w)$ with different entries satisfying $\emptyset\ne\dd K_w\cap K_{(w^{(1)})^-}\subseteq\dd K_w\cap K_{wj_1}$ and $\emptyset\ne\dd K_w\cap K_{w^{(m)}}\subseteq\dd K_w\cap K_{wj_2}$. Noting that
$$g_\mu((w^{(1)})^-)=\mu^{k-1}\le2\mu^k\le\sum_{i=1}^j\mu^k=\sum_{i=1}^jg_\mu(w^{(i)}),$$
we have
$$g_\mu((w^{(1)})^-)+\sum_{i=j+1}^mg_\mu(w^{(i)})\le\sum_{i=1}^mg_\mu(w^{(i)}).$$
Moreover, we have $|(w^{(1)})^-|=k-1<k$.

If $|w^{(m)}|=k$, then by similar argument to the above, we have another chain going through $K_w$ with different entries and with the last element $(w^{(m)})^-$ satisfying $|(w^{(m)})^-|=k-1<k$.

For a possibly new chain, denoted by $\myset{v^{(1)},\ldots,v^{(l)}}$, that satisfies $|v^{(1)}|<k$ and $|v^{(l)}|<k$. If
$$\max\myset{|v^{(i)}|:i=1,\ldots,l}<k,$$
then this is our desired chain. Otherwise, let
$$j=\min\myset{j:|v^{(j)}|=k}.$$
By similar argument to the above, let
$$p=\max\myset{p:|v^{(j)}|=\ldots=|v^{(p)}|,(v^{(j)})^-=\ldots=(v^{(p)})^-},$$
then $l-1\ge p\ge j+1$. Hence we have $\myset{v^{(1)},\ldots,v^{(j-1)},(v^{(j)})^-,v^{(p+1)},\ldots,v^{(l)}}\in\calC(K_w)$ with different entries satisfying
$$\sum_{i=1}^{j-1}g_\mu(v^{(i)})+g_\mu((v^{(j)})^-)+\sum_{i=p+1}^lg_\mu(v^{(i)})\le\sum_{i=1}^lg_\mu(v^{(i)}).$$

Repeating the above consideration finitely many times, we eventually obtain the desired chain.

By induction principle, we have the desired result.
\end{proof}

\begin{myrmk}
By the above proof, $1/2$ is critically important.
\end{myrmk}

\begin{mycor}\label{cor_metric}
For all $n\ge1$, for all $w^{(1)},w^{(2)}\in W_n$. If $K_{w^{(1)}}\cap K_{w^{(2)}}=\emptyset$, then
$$d_\mu(K_{w^{(1)}},K_{w^{(2)}}):=\inf\myset{d_\mu(x,y):x\in K_{w^{(1)}},y\in K_{w^{(2)}}}\ge\mu^n.$$
\end{mycor}

\begin{proof}
For all $x\in K_{w^{(1)}}$, $y\in K_{w^{(2)}}$, for all $\myset{v^{(1)},\ldots,v^{(m)}}\in\calC(x,y)$, there exists $w^{(3)}\in W_n$ with $w^{(3)}\ne w^{(1)}$ and $w^{(3)}\ne w^{(2)}$, either there exists $i=1,\ldots,m$ such that $K_{v^{(i)}}\supseteq K_{w^{(3)}}$ or there exist $i_1\le i_2$ such that $\myset{v^{(i_1)},\ldots,v^{(i_2)}}\in\calC(K_{w^{(3)}})$ with different entries.

For the first case, we have
$$\sum_{i=1}^mg_\mu(v^{(i)})\ge g_\mu(v^{(i)})=\mu^{|v^{(i)}|}\ge\mu^{|w^{(3)}|}=\mu^n.$$

For the second case, by Lemma \ref{lem_metric_pre1}, we have
$$\sum_{i=1}^mg_\mu(v^{(i)})\ge\sum_{i=i_1}^{i_2}g_\mu(v^{(i)})\ge\mu^{|w^{(3)}|}=\mu^n.$$

Hence
$$d_\mu(x,y)\ge\mu^n,$$
hence
$$d_\mu(K_{w^{(1)}},K_{w^{(2)}})\ge\mu^n.$$

\end{proof}

\begin{mylem}\label{lem_metric_pre2}
For all $w\in W_*$, for all $x,y\in K_w$, we have
\begin{footnotesize}
\begin{equation*}
d_\mu(x,y)=\inf\myset{\sum_{i=1}^mg_\mu(w^{(i)}):\myset{w^{(1)},\ldots,w^{(m)}}\in\calC(x,y),K_{w^{(i)}}\subseteq K_w\text{ for all }i=1,\ldots,m}.
\end{equation*}
\end{footnotesize}
\end{mylem}

\begin{proof}
If $w=\emptyset$, then this result is trivial. We may assume that $|w|\ge1$. It is obvious that the LHS $\le$ the RHS. Since $\myset{w}\in\calC(x,y)$, we have the RHS $\le\mu^{|w|}$.

We only need to show that for arbitrary $\myset{w^{(1)},\ldots,w^{(m)}}\in\calC(x,y)$, we have
$$\sum_{i=1}^mg_\mu(w^{(i)})\ge\text{RHS}.$$
If there exists $i=1,\ldots,m$ such that $|w^{(i)}|\le|w|$, then
$$\sum_{i=1}^mg_\mu(w^{(i)})\ge g_\mu(w^{(i)})=\mu^{|w^{(i)}|}\ge\mu^{|w|}\ge\text{RHS}.$$
We may assume that $|w^{(i)}|\ge|w|+1$ for all $i=1,\ldots,m$.

If $K_{w^{(i)}}\subseteq K_w$ for all $i=1,\ldots,m$, then it is trivial to have
$$\sum_{i=1}^mg_\mu(w^{(i)})\ge\text{RHS}.$$
Otherwise, there exists $i=1,\ldots,m$ such that $K_{w^{(i)}}\not\subseteq K_w$.

Then there exists $v\in W_{|w|}$ with $v\ne w$ and $K_w\cap K_v\ne\emptyset$, there exist $i_1\le i_2$ such that $K_{w^{(i)}}\subseteq K_v$ for all $i=i_1,\ldots,i_2$ and exact one of the following conditions holds.
\begin{enumerate}[(a)]
\item $i_2=m$.
\item $i_1=1$, $i_2<m$ and $K_{w^{(i_2+1)}}\subseteq K_w$.
\item $i_1=1$, $i_2<m$ and $K_{w^{(i_2+1)}}\subseteq K_u$ for some $u\in W_{|w|}$ with $u\ne w$ and $u\ne v$.
\item $i_1>1$, $K_{w^{(i_1-1)}}\subseteq K_w$, $i_2<m$ and $K_{w^{(i_2+1)}}\subseteq K_w$.
\item $i_1>1$, $K_{w^{(i_1-1)}}\subseteq K_w$, $i_2<m$ and $K_{w^{(i_2+1)}}\subseteq K_u$ for some $u\in W_{|w|}$ with $u\ne w$ and $u\ne v$.
\end{enumerate}

For (c) and (e). We have $\myset{w^{(i_1)},\ldots,w^{(i_2)}}\in\calC(K_v)$ with different entries. By Lemma \ref{lem_metric_pre1}, we have
$$\sum_{i=1}^mg_\mu(w^{(i)})\ge\sum_{i=i_1}^{i_2}g_\mu(w^{(i)})\ge\mu^{|v|}=\mu^{|w|}\ge\text{RHS}.$$

For (a), (b) and (d). By reflection, we replace $w^{(i_1)},\ldots,w^{(i_2)}$ by $v^{(i_1)},\ldots,v^{(i_2)}$ that are symmetric about $K_w\cap K_v$, see Figure \ref{fig_reflection}, then $K_{v^{(i_1)}},\ldots,K_{v^{(i_2)}}\subseteq K_w$ and $g_\mu(v^{(i)})=g_\mu(w^{(i)})$ for all $i=i_1,\ldots,i_2$.

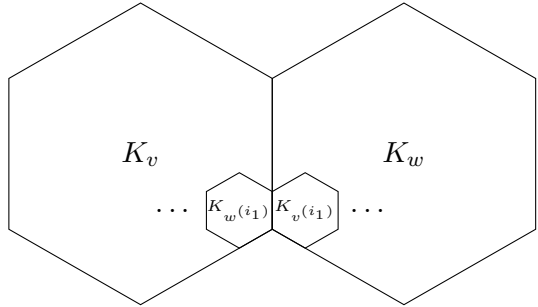
\begin{figure}[ht]
\centering
\begin{tikzpicture}[scale=0.5]
\draw (0,0)--(0,4)--(2*1.73205080756887,6)--(4*1.73205080756887,4)--(4*1.73205080756887,0)--(2*1.73205080756887,-2)--cycle;
\draw (0,0)--(0,4)--(-2*1.73205080756887,6)--(-4*1.73205080756887,4)--(-4*1.73205080756887,0)--(-2*1.73205080756887,-2)--cycle;

\draw (0,0)--(0,1)--(0.5*1.73205080756887,6/4)--(4/4*1.73205080756887,4/4)--(4/4*1.73205080756887,0)--(2/4*1.73205080756887,-2/4)--cycle;
\draw (0,0)--(0,1)--(-0.5*1.73205080756887,6/4)--(-4/4*1.73205080756887,4/4)--(-4/4*1.73205080756887,0)--(-2/4*1.73205080756887,-2/4)--cycle;

\draw (2*1.73205080756887,2) node {{{$K_w$}}};
\draw (-2*1.73205080756887,2) node {{{$K_v$}}};

\draw (2/4*1.73205080756887,2/4) node {{\tiny{$K_{v^{(i_1)}}$}}};
\draw (-2/4*1.73205080756887,2/4) node {{\tiny{$K_{w^{(i_1)}}$}}};

\draw (2/4*1.73205080756887+1.7,2/4) node {\ldots};
\draw (-2/4*1.73205080756887-1.7,2/4) node {\ldots};
\end{tikzpicture}
\caption{The Reflection}\label{fig_reflection}
\end{figure}

Repeat the above consideration to the chain
$$\myset{w^{(1)},\ldots,w^{(i_1-1)},v^{(i_1)},\ldots,v^{(i_2)},w^{(i_2+1)},\ldots,w^{(m)}}$$
finitely many times, exact one of the following cases occurs.
\begin{enumerate}[(i)]
\item We obtain a chain denoted by $\myset{v^{(1)},\ldots,v^{(m)}}\in\calC(x,y)$ with $K_{v^{(i)}}\subseteq K_w$ for all $i=1,\ldots,m$ and $\sum_{i=1}^mg_\mu(v^{(i)})=\sum_{i=1}^mg_\mu(w^{(i)})$.
\item Either (c) or (e) holds.
\end{enumerate}
For (i), we have
$$\sum_{i=1}^mg_\mu(w^{(i)})=\sum_{i=1}^mg_\mu(v^{(i)})\ge\text{RHS}.$$
For (ii), we have
$$\sum_{i=1}^mg_\mu(w^{(i)})\ge\mu^{|w|}\ge\text{RHS}.$$

Hence, we have the LHS $\ge$ the RHS.
\end{proof}

\begin{proof}[Proof of Theorem \ref{thm_metric}]
The case $\mu\in(0,1/2)$ has been considered in Lemma \ref{lem_notmetric}. We may assume that $\mu\in[1/2,1)$.

We only need to show that for arbitrary fixed $x,y\in K$ with $x\ne y$, we have $d_\mu(x,y)>0$.

Since $\pi^{-1}(x)$ contains at most two elements in $W_\infty$ for all $x\in K$, there exist unique $w\in W_*$ and $j_1,j_2\in W$ with $j_1\ne j_2$ such that $x\in K_{wj_1}\backslash K_{wj_2}$ and $y\in K_{wj_2}\backslash K_{wj_1}$.

If $K_{wj_1}\cap K_{wj_2}=\emptyset$, then by Corollary \ref{cor_metric}, we have
$$d_\mu(x,y)\ge d_\mu(K_{wj_1},K_{wj_2})\ge\mu^{|w|+1}>0.$$

If $K_{wj_1}\cap K_{wj_2}\ne\emptyset$, then without lose of generality, we may assume that $j_1=0$ and $j_2=1$, then
$$K_{w0}\cap K_{w1}=\pi\left(\myset{w01v\sim w11v,w02v\sim w12v:v\in\myset{0,5}^\infty}\right),$$
there exist $k^{(1)},k^{(2)}\in W$, $v^{(1)}=v^{(1)}_1v^{(1)}_2\ldots,v^{(2)}=v^{(2)}_1v^{(2)}_2\ldots\in W_\infty$ such that $w0k^{(1)}v^{(1)}\in\pi^{-1}(x)$, $w1k^{(2)}v^{(2)}\in\pi^{-1}(y)$.

If $k^{(1)}\ne k^{(2)}$ or $k^{(1)}\in\myset{0,3,4,5}$ or $k^{(2)}\in\myset{0,3,4,5}$, then for all $\myset{w^{(1)},\ldots,w^{(m)}}\in\calC(x,y)$, either there exists $i=1,\ldots,m$ such that $K_{w^{(i)}}$ contains a $(|w|+2)$-cell or there exists some sub-chain passing through a $(|w|+2)$-cell with different entries, hence
$$\sum_{i=1}^{m}g_\mu(w^{(i)})\ge\mu^{|w|+2},$$
hence
$$d_\mu(x,y)\ge\mu^{|w|+2}>0.$$

Hence we may assume that $k^{(1)}=k^{(2)}\in\myset{1,2}$, without lose of generality, we may assume that $k^{(1)}=k^{(2)}=1$.

Since $x\in K_{w0}\backslash K_{w1}$ and $y\in K_{w1}\backslash K_{w0}$, we have $v^{(1)},v^{(2)}\not\in\myset{0,5}^\infty$. Let
\begin{align*}
n^{(1)}&=\min\myset{n:v^{(1)}_n\not\in\myset{0,5}},\\
n^{(2)}&=\min\myset{n:v^{(2)}_n\not\in\myset{0,5}}.
\end{align*}
For all $\myset{w^{(1)},\ldots,w^{(m)}}\in\calC(x,y)$, for all $j=1,2$, either there exists $i=1,\ldots,m$ such that $K_{w^{(i)}}$ contains a $(|w|+2+n^{(j)})$-cell or there exists some sub-chain passing through a $(|w|+2+n^{(j)})$-cell with different entries, hence
$$\sum_{i=1}^mg_\mu(w^{(i)})\ge\mu^{|w|+2+n^{(1)}}+\mu^{|w|+2+n^{(2)}},$$
hence
$$d_\mu(x,y)\ge\mu^{|w|+2+n^{(1)}}+\mu^{|w|+2+n^{(2)}}>0.$$

Therefore, we have $d_\mu(x,y)>0$ for all $x,y\in K$ with $x\ne y$.

For all $j\in W$, for all $x,y\in K$, we have
\begin{footnotesize}
\begin{align*}
&d_\mu(f_j(x),f_j(y))\\
&=\inf\myset{\sum_{i=1}^mg_\mu(w^{(i)}):\myset{w^{(1)},\ldots,w^{(m)}}\in\calC(f_j(x),f_j(y)),K_{w^{(i)}}\subseteq K_j\text{ for all }i=1,\ldots,m}\\
&=\inf\myset{\sum_{i=1}^mg_\mu(jw^{(i)}):\myset{jw^{(1)},\ldots,jw^{(m)}}\in\calC(f_j(x),f_j(y)),K_{jw^{(i)}}\subseteq K_j\text{ for all }i=1,\ldots,m}\\
&=\mu\inf\myset{\sum_{i=1}^mg_\mu(w^{(i)}):\myset{w^{(1)},\ldots,w^{(m)}}\in\calC(x,y)}=\mu d_\mu(x,y),
\end{align*}
\end{footnotesize}
where we use Lemma \ref{lem_metric_pre2} in the first equality, we use the fact that
$$\myset{jw^{(1)},\ldots,jw^{(m)}}\in\calC(f_j(x),f_j(y))$$
if and only if
$$\myset{w^{(1)},\ldots,w^{(m)}}\in\calC(x,y)$$
in the third equality.

For all $x,y\in K$, since $\myset{\emptyset}\in\calC(x,y)$, we have
$$d_\mu(x,y)\le g_\mu(\emptyset)=1,$$
hence $\mathrm{diam}_\mu(K)\le1$.

For all $x\in K_0$, $y\in K_3$, for all $\myset{w^{(1)},\ldots,w^{(m)}}\in\calC(x,y)$.

Denote
\begin{enumerate}[(a)]
\item Either there exists $i=1,\ldots,m$ such that $K_{w^{(i)}}\supseteq K_1$ or there exist $i_1\le i_2$ such that $\myset{w^{(i_1)},\ldots,w^{(i_2)}}\in\calC(K_1)$ with different entries.
\item Either there exists $i=1,\ldots,m$ such that $K_{w^{(i)}}\supseteq K_2$ or there exist $i_1\le i_2$ such that $\myset{w^{(i_1)},\ldots,w^{(i_2)}}\in\calC(K_2)$ with different entries.
\item Either there exists $i=1,\ldots,m$ such that $K_{w^{(i)}}\supseteq K_4$ or there exist $i_1\le i_2$ such that $\myset{w^{(i_1)},\ldots,w^{(i_2)}}\in\calC(K_4)$ with different entries.
\item Either there exists $i=1,\ldots,m$ such that $K_{w^{(i)}}\supseteq K_5$ or there exist $i_1\le i_2$ such that $\myset{w^{(i_1)},\ldots,w^{(i_2)}}\in\calC(K_5)$ with different entries.
\end{enumerate}
Then either (a) and (b) hold or (c) and (d) hold. In both cases, we have
$$\sum_{i=1}^mg_\mu(w^{(i)})\ge\mu+\mu=2\mu\ge1,$$
hence $d_\mu(x,y)\ge1$, hence $\mathrm{diam}_\mu(K)=1$. By the contraction property of $f_0,\ldots,f_5$, we have $\mathrm{diam}_\mu(K_w)=\mu^{|w|}$.

By Lemma \ref{lem_metric_basic} and Corollary \ref{cor_metric}, we have the conditions in \cite[Theorem 1.5.7]{Kig01} hold, hence the Hausdorff dimension of $(K,d_\mu)$ is $\alpha=-\log6/\log\mu$, the normalized Hausdorff measure $\nu$ of dimension $\alpha$ exists and is given by a self-similar measure.
\end{proof}

\section{Proof of Proposition \ref{prop_chain}}\label{sec_chain}

Recall that a metric space $(K,d)$ satisfies the chain condition or the $\theta$-chain condition if there exists a positive constant $C$ such that for all $x,y\in K$, for all $n\ge1$, there exists a sequence $\myset{x_0,x_1,\ldots,x_n}$ in $K$ with $x_0=x$ and $x_n=y$ such that
$$d(x_i,x_{i+1})\le C\frac{d(x,y)}{n}\text{ for all }i=0,\ldots,n-1,$$
or
\begin{equation}\label{eqn_thetachain}
d(x_i,x_{i+1})\le C\frac{d(x,y)}{n^\theta}\text{ for all }i=0,\ldots,n-1.
\end{equation}

For all $n\ge1$, let $G_n$ be the graph with vertex set $W_n$ and edge set given by
$$\myset{(w^{(1)},w^{(2)}):w^{(1)},w^{(2)}\in W_n,w^{(1)}\ne w^{(2)},K_{w^{(1)}}\cap K_{w^{(2)}}\ne\emptyset}.$$
For all $w^{(1)},w^{(2)}\in W_n$, we denote $w^{(1)}\sim_nw^{(2)}$ if $(w^{(1)},w^{(2)})$ is an edge in $G_n$. Let $d_n$ be the graph metric on $G_n$, that is, $d_n(w^{(1)},w^{(2)})$ is the minimum of the lengths of all paths joining $w^{(1)}$ and $w^{(2)}$. Denote the diameter of $G_n$ as
$$\mathrm{diam}(G_n):=\sup\myset{d_n(w^{(1)},w^{(2)}):w^{(1)},w^{(2)}\in W_n}.$$

\begin{mylem}\label{lem_diam}
There exists some positive constant $C$ such that for all $n\ge1$, we have
$$\frac{1}{C}(n\cdot 2^n)\le\mathrm{diam}(G_n)\le C(n\cdot2^n).$$
\end{mylem}

\begin{proof}
For arbitrary fixed $n\ge1$. Obviously, $G_n$ is a planer graph. Denote the outer circumference path $\mathrm{Out}_n$ as in \cite[Definition 5.1]{BKNPPT12}. By \cite[Proposition 5.2]{BKNPPT12}, we have
$$|\mathrm{Out}_n|=3n\cdot2^n.$$

For all $w\in W_n$, we have
$$d_n(w,\mathrm{Out}_n):=\inf\myset{d_n(w,v):v\in\mathrm{Out}_n}\lesssim n\cdot2^n.$$
For all $w^{(1)},w^{(2)}\in W_n$, we have
$$d_n(w^{(1)},w^{(2)})\le d_n(w^{(1)},\mathrm{Out}_n)+d_n(w^{(2)},\mathrm{Out}_n)+|\mathrm{Out}_n|\lesssim n\cdot2^n,$$
hence
$$\mathrm{diam}(G_n)\lesssim n\cdot2^n.$$

By the graph structure of $G_n$, there exists some positive constant $c$ such that for all $n\ge1$
$$\mathrm{diam}(G_{n+1})\ge2\mathrm{diam}(G_n)+c2^n.$$
By recursion, we have
$$\mathrm{diam}(G_n)\gtrsim n\cdot2^n.$$
Therefore, we have
$$\mathrm{diam}(G_n)\asymp n\cdot2^n.$$
\end{proof}

\begin{myrmk}
It was conjectured in \cite[Conjecture 5.4]{BKNPPT12} an explicit formula for $\mathrm{diam}(G_n)$.
\end{myrmk}

\begin{proof}[Proof of Proposition \ref{prop_chain}]
Suppose that $d_\mu$ satisfies the chain condition. Let $C$ be the constant in the definition of the chain condition, take $k_1\ge1$ satisfying $\mu^{-k_1}>C$, let $c$ be the constant in Lemma \ref{lem_diam}.

For all $k>2c\mu^{-k_1}$. Take $w,v\in W_{k}$ such that $d_{k}(w,v)=\mathrm{diam}(G_{k})$, take $x\in K_w$, $y\in K_v$, then there exists a sequence $\myset{x_0,\ldots,x_{\lceil \mu^{-(k+k_1)}\rceil}}$ in $K$ with $x_0=x$ and $x_{\lceil\mu^{-(k+k_1)}\rceil}=y$ such that
$$d_\mu(x_i,x_{i+1})\le C\frac{d_\mu(x,y)}{\lceil\mu^{-(k+k_1)}\rceil}\le\frac{C}{\mu^{-(k+k_1)}}<\mu^k.$$
Take $w^{(0)},\ldots,w^{(\lceil\mu^{-(k+k_1)}\rceil)}\in W_{k}$ with $w^{(0)}=w$, $w^{(\lceil\mu^{-(k+k_1)}\rceil)}=v$ and $x_i\in K_{w{(i)}}$ for all $i=0,\ldots,\lceil\mu^{-(k+k_1)}\rceil$.

For all $i=0,\ldots,\lceil\mu^{-(k+k_1)}\rceil-1$, we have $K_{w^{(i)}}\cap K_{w^{(i+1)}}\ne\emptyset$, otherwise, by Corollary \ref{cor_metric}, we have
$$d_\mu(x_i,x_{i+1})\ge d_\mu(K_{w^{(i)}},K_{w^{(i+1)}})\ge\mu^{k},$$
contradiction! Hence for all $i=0,\ldots,\lceil\mu^{-(k+k_1)}\rceil-1$, either $w^{(i)}=w^{(i+1)}$ or $w^{(i)}\sim_{k}w^{(i+1)}$. Hence
$$\mathrm{diam}(G_{k})=d_{k}(w,v)=d_{k}(w^{(0)},w^{(2^{k+k_1})})\le\lceil\mu^{-(k+k_1)}\rceil.$$
By Lemma \ref{lem_diam}, we have $\mathrm{diam}(G_{k})\ge\frac{1}{c}(k\cdot2^k)$, hence
$$\frac{1}{c}(k\cdot2^k)\le\lceil\mu^{-(k+k_1)}\rceil\le2\mu^{-(k+k_1)},$$
that is,
$$k\le\frac{2c}{(2\mu)^k}\mu^{-k_1}\le 2c\mu^{-k_1},$$
contradiction!

We only need to show that Equation (\ref{eqn_thetachain}) holds for a sequence $\myset{n_k}_{k\ge1}$ with
$$\sup_{k\ge1}\frac{n_{k+1}}{n_k}<+\infty$$
for all $x,y\in K$ with $d_\mu(x,y)<1/2$.

Let $c$ be the constant in Lemma \ref{lem_diam}.

Let $n_k=2([ck2^k]+1)$. It is obvious that $\sup_{k\ge1}n_{k+1}/n_k<+\infty$.

For all $x,y\in K$ with $x\ne y$ and $d_\mu(x,y)<1/2$, there exists some integer $N\ge1$ such that
$$\frac{1}{2^{N+1}}\le d_\mu(x,y)<\frac{1}{2^{N}}.$$
There exist $w,v\in W_N$ such that $x\in K_w,y\in K_v$, then $K_{w}\cap K_v\ne\emptyset$, otherwise, by Corollary \ref{cor_metric}, we have
$$d_\mu(x,y)\ge d_\mu(K_w,K_v)\ge\mu^N\ge\frac{1}{2^N},$$
contradiction!

Since $\mathrm{diam}(G_k)\le ck2^k$ by Lemma \ref{lem_diam}, there exist $w^{(0)},\ldots,w^{(n_k)}\in W_k$ with $x\in K_{ww^{(0)}}$, $y\in K_{vw^{(n_k)}}$ satisfying
$$K_{ww^{(n_k/2)}}\cap K_{vw^{(n_k/2+1)}}\ne\emptyset,$$
$$K_{ww^{(i)}}\cap K_{ww^{(i+1)}}\ne\emptyset\text{ for all }i=0,\ldots,\frac{n_k}{2}-1,$$
$$K_{vw^{(i)}}\cap K_{vw^{(i+1)}}\ne\emptyset\text{ for all }i=\frac{n_k}{2}+1,\ldots,n_k-1.$$
Take arbitrary $x_i\in K_{ww^{(i)}}$ for all $i=1,\ldots,n_k/2$ and $x_i\in K_{vw^{(i)}}$ for all $i=n_k/2+1,\ldots,n_k$, then
$$d_\mu(x_i,x_{i+1})\le2\mu^{N+k}\le\frac{2}{\mu}\mu^kd_\mu(x,y).$$
Take a constant $C$ satisfying
$$2^{1+2\theta}c^\theta k^\theta\le C\mu\frac{1}{(2^\theta\mu)^k}\text{ for all }k\ge1,$$
then
$$d_\mu(x_i,x_{i+1})\le\frac{2}{\mu}\mu^kd_\mu(x,y)\le C\frac{d_\mu(x,y)}{n_k^\theta}\text{ for all }i=0,\ldots,n_k-1.$$
\end{proof}

\section{Proof of Theorem \ref{thm_Harnack}}\label{sec_Harnack}

The following result states that an $n$-cell is comparable to a ball with radius $\mu^{n}$ with respect to the intrinsic metric $d_\mu$.

\begin{myprop}\label{prop_cellandball}
For all $n\ge0$, for all $w\in W_n$, we have the following results.
\begin{enumerate}[(1)]
\item For all $x\in K_w$, we have $K_w\subseteq B_\mu(x,2\mu^n)$.
\item There exists $x\in K_w$ such that $B_\mu(x,\mu^{n+2})\subseteq K_w$.
\end{enumerate}
\end{myprop}

\begin{proof}
(1) Since $\mathrm{diam}_\mu(K_w)=\mu^{n}$, for all $x\in K_w$, we have
$$K_w\subseteq B_\mu(x,2\mathrm{diam}_\mu(K_w))=B_\mu(x,2\mu^{n}).$$

(2) Take $x\in K_{w22}\subseteq\mathrm{int}(K_w)$, see \cite[FIGURE 2]{BKNPPT12}, for all $v\in W_{n+2}$ with $K_{w22}\cap K_v=\emptyset$, by Corollary \ref{cor_metric}, we have $d_\mu(K_{w22},K_v)\ge\mu^{n+2}$. In particular, for all $y\not\in K_w$, we have $d_\mu(x,y)\ge\mu^{n+2}$, hence $B_\mu(x,\mu^{n+2})\subseteq K_w$.
\end{proof}

For all $n\ge0$, let $X^{(n)}$ be the simple random walk on $H_n$, let $\tau_B$ be the first exit time of $X^{(n)}$ from a subset $B$ of $V_n$.

We use knight move technique developed by Barlow and Bass \cite{BB89}. We need do some preparations.

First, we have corner move as follows. 

\begin{mylem}\label{lem_corner}
For all $n\ge1$, for all $w^{(1)},w^{(2)}\in W_n$ with $w^{(1)}\ne w^{(2)}$ and $K_{w^{(1)}}\cap K_{w^{(2)}}\ne\emptyset$. Each of $\dd K_{w^{(1)}},\dd K_{w^{(2)}}$ consists of six disjoint parts, $\dd K_{w^{(1)}}\cup\dd K_{w^{(2)}}$ consists of ten disjoint parts, $\dd K_{w^{(1)}}\cap\dd K_{w^{(2)}}$ consists of two disjoint parts. Denote $L_0$ as one part of $\dd K_{w^{(1)}}\cap\dd K_{w^{(2)}}$, denote $L_1,\ldots,L_8$ as the eight parts of $(\dd K_{w^{(1)}}\backslash\dd K_{w^{(2)}})\cup(\dd K_{w^{(2)}}\backslash\dd K_{w^{(1)}})$, where $L_1,L_8$ are two parts adjacent to $L_0$. Let $B=(K_{w^{(1)}}\cup K_{w^{(2)}})\backslash(L_1\cup\ldots\cup L_8)$, see Figure \ref{fig_corner}. Then for all $k\ge n$, for all $x\in L_0\cap V_k$, we have
$$\mathbb{P}_x\left[X^{(k)}_{\tau_B}\in L_1\right]\ge\frac{1}{8}.$$

\begin{figure}[ht]
\centering

\begin{tikzpicture}[scale=1]

\draw (0,0)--(1,1.73205080756887)--(2,0)--(1,-1.73205080756887)--cycle;
\draw (0,0)--(2,0);

\draw[very thick] (0,0)--(1,0);
\draw[very thick] (0,0)--(1/2,1/2*1.73205080756887);

\draw[fill=black] (0,0) circle (0.03);
\draw[fill=black] (1/2,1/2*1.73205080756887) circle (0.03);
\draw[fill=black] (1,1.73205080756887) circle (0.03);
\draw[fill=black] (1.5,1/2*1.73205080756887) circle (0.03);
\draw[fill=black] (2,0) circle (0.03);
\draw[fill=black] (1.5,-0.5*1.73205080756887) circle (0.03);
\draw[fill=black] (1,-1.73205080756887) circle (0.03);
\draw[fill=black] (0.5,-0.5*1.73205080756887) circle (0.03);

\draw[fill=black] (1,0) circle (0.03);

\draw (0.6,0.3) node {$L_0$};

\draw (-0.1,0.5) node {$L_1$};
\draw (0.4,0.5+0.5*1.73205080756887) node {$L_2$};
\draw (1.6,0.5+0.5*1.73205080756887) node {$L_3$};
\draw (2.1,0.5) node {$L_4$};
\draw (2.1,-0.5) node {$L_5$};
\draw (1.6,-0.5-0.5*1.73205080756887) node {$L_6$};
\draw (0.4,-0.5-0.5*1.73205080756887) node {$L_7$};
\draw (-0.1,-0.5) node {$L_8$};
\end{tikzpicture}
\caption{Corner Move}\label{fig_corner}
\end{figure}
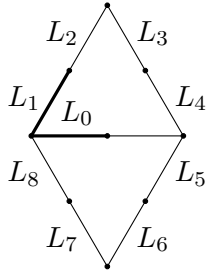
\end{mylem}

Second, we have knight move \Rmnum{1} as follows.

\begin{mylem}\label{lem_knight1}
For all $n\ge0$, for all $w\in W_n$. $\dd K_w\cap(\dd K_{w0}\cup\ldots\cup\dd K_{w5})$ consists of twelve disjoint parts, $\dd K_{w0}\cap\dd K_{w1}$ consists of two disjoint parts. Denote $L_0$ as one part of $\dd K_{w0}\cap\dd K_{w1}$ which is not adjacent to $\dd K_w$, denote $L_1,\ldots,L_{12}$ as the twelve parts of $\dd K_w\cap(\dd K_{w0}\cup\ldots\cup\dd K_{w5})$, where $L_1,L_{12}$ are two parts adjacent to $\dd K_{w0}\cap\dd K_{w1}$. Let $B=\mathrm{int}(K_w)$, see Figure \ref{fig_knight1}. Then for all $k\ge n$, for all $x\in L_0\cap V_k$, we have
$$\mathbb{P}_x\left[X^{(k)}_{\tau_B}\in L_1\right]\ge\frac{1}{12}.$$

\begin{figure}[ht]
\centering

\begin{tikzpicture}[scale=1]

\draw (1,1.73205080756887)--(2,0)--(1,-1.73205080756887)--(-1,-1.73205080756887)--(-2,0)--(-1,1.73205080756887)--cycle;
\draw (-2,0)--(2,0);
\draw (1,1.73205080756887)--(-1,-1.73205080756887);
\draw (-1,1.73205080756887)--(1,-1.73205080756887);

\draw [very thick] (0,0)--(1,0);
\draw [very thick] (2,0)--(1.5,0.5*1.73205080756887);

\draw[fill=black] (0,0) circle (0.03);
\draw[fill=black] (1,0) circle (0.03);

\draw[fill=black] (2,0) circle (0.03);
\draw[fill=black] (1.5,0.5*1.73205080756887) circle (0.03);
\draw[fill=black] (1,1.73205080756887) circle (0.03);
\draw[fill=black] (0,1.73205080756887) circle (0.03);
\draw[fill=black] (-1,1.73205080756887) circle (0.03);
\draw[fill=black] (-1.5,0.5*1.73205080756887) circle (0.03);
\draw[fill=black] (-2,0) circle (0.03);
\draw[fill=black] (-1.5,-0.5*1.73205080756887) circle (0.03);
\draw[fill=black] (-1,-1.73205080756887) circle (0.03);
\draw[fill=black] (0,-1.73205080756887) circle (0.03);
\draw[fill=black] (1,-1.73205080756887) circle (0.03);
\draw[fill=black] (1.5,-1/2*1.73205080756887) circle (0.03);

\draw (0.5,0.2) node {$L_0$};

\draw (2.1,0.5) node {$L_1$};
\draw (1.5,0.5+0.5*1.73205080756887) node {$L_2$};
\draw (0.5,2) node {$L_3$};

\draw (-2.1,0.5) node {$L_6$};
\draw (-1.5,0.5+0.5*1.73205080756887) node {$L_5$};
\draw (-0.5,2) node {$L_4$};

\draw (-2.1,-0.5) node {$L_7$};
\draw (-1.5,-0.5-0.5*1.73205080756887) node {$L_8$};
\draw (-0.5,-2) node {$L_9$};

\draw (2.1,-0.5) node {$L_{12}$};
\draw (1.5,-0.5-0.5*1.73205080756887) node {$L_{11}$};
\draw (0.5,-2) node {$L_{10}$};

\end{tikzpicture}
\caption{Knight Move \Rmnum{1}}\label{fig_knight1}
\end{figure}
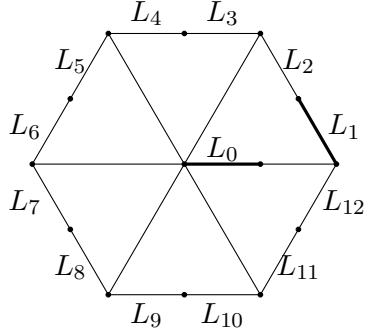
\end{mylem}

Third, we have knight move \Rmnum{2} as follows.

\begin{mylem}\label{lem_knight2}
For all $n\ge1$, for all $w^{(1)},w^{(2)}\in W_n$ with $w^{(1)}\ne w^{(2)}$ and $K_{w^{(1)}}\cap K_{w^{(2)}}\ne\emptyset$, there exist $i^{(1)},i^{(2)},j^{(1)},j^{(2)}\in W$ with $i^{(1)}\ne i^{(2)}$ and $j^{(1)}\ne j^{(2)}$ such that $K_{w^{(1)}i^{(1)}}\cap K_{w^{(2)}j^{(1)}}\ne\emptyset$, $K_{w^{(2)}j^{(1)}}\cap K_{w^{(2)}j^{(2)}}\ne\emptyset$, $K_{w^{(1)}i^{(2)}}\cap K_{w^{(2)}j^{(2)}}\ne\emptyset$ and $K_{w^{(1)}i^{(1)}}\cap K_{w^{(1)}i^{(2)}}\ne\emptyset$. Let $v^{(1)}=w^{(1)}i^{(1)}$, $v^{(2)}=w^{(2)}j^{(1)}$, $v^{(3)}=w^{(2)}j^{(2)}$ and $v^{(4)}=w^{(1)}i^{(2)}$.

$\cup_{k=1}^4(\dd K_{v^{(k)}}\backslash(\cup_{l\ne k}\dd K_{v^{(l)}}))$ consists of eight disjoint parts, $\dd K_{v^{(1)}}\cap\dd K_{v^{(2)}}$ consists of two disjoint parts. Denote $L_0$ as one part of $\dd K_{v^{(1)}}\cap\dd K_{v^{(2)}}$ which is not adjacent to $\cup_{k=1}^4(\dd K_{v^{(k)}}\backslash(\cup_{l\ne k}\dd K_{v^{(l)}}))$, denote $L_1,\ldots,L_8$ as the eight parts of

\noindent $\cup_{k=1}^4(\dd K_{v^{(k)}}\backslash(\cup_{l\ne k}\dd K_{v^{(l)}}))$, where $L_1,L_8$ are two parts adjacent to $\dd K_{v^{(1)}}\cap\dd K_{v^{(2)}}$. Let $B=(\cup_{k=1}^4K_{v^{(k)}})\backslash(\cup_{k=1}^8L_k)$, see Figure \ref{fig_knight2}. Then for all $k\ge n$, for all $x\in L_0\cap V_k$, we have
$$\mathbb{P}_x\left[X^{(k)}_{\tau_B}\in L_1\right]\ge\frac{1}{8}.$$
\begin{figure}[ht]
\centering

\begin{tikzpicture}[scale=1]

\draw (-2,0)--(0,2)--(2,0)--(0,-2)--cycle;
\draw (-2,0)--(2,0);
\draw (0,-2)--(0,2);

\draw[very thick] (0,0)--(1,0);
\draw[very thick] (2,0)--(1,1);

\draw[fill=black] (0,0) circle (0.03);
\draw[fill=black] (1,0) circle (0.03);

\draw[fill=black] (2,0) circle (0.03);
\draw[fill=black] (1,1) circle (0.03);
\draw[fill=black] (0,2) circle (0.03);
\draw[fill=black] (-1,1) circle (0.03);
\draw[fill=black] (-2,0) circle (0.03);
\draw[fill=black] (-1,-1) circle (0.03);
\draw[fill=black] (0,-2) circle (0.03);
\draw[fill=black] (1,-1) circle (0.03);

\draw (0.5,0.2) node {$L_0$};

\draw (1.75,0.7) node {$L_1$};
\draw (0.75,1.7) node {$L_2$};

\draw (-1.75,0.7) node {$L_4$};
\draw (-0.75,1.7) node {$L_3$};

\draw (-1.75,-0.7) node {$L_5$};
\draw (-0.75,-1.7) node {$L_6$};

\draw (1.75,-0.7) node {$L_8$};
\draw (0.75,-1.7) node {$L_7$};

\end{tikzpicture}
\caption{Knight Move \Rmnum{2}}\label{fig_knight2}
\end{figure}
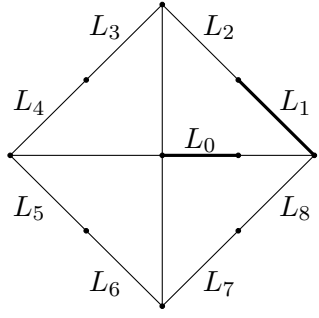
\end{mylem}

\begin{proof}[Proof of Lemma \ref{lem_corner}, Lemma \ref{lem_knight1} and Lemma \ref{lem_knight2}]
Denote
$$p_i=\bbP_x\left[X^{(k)}_{\tau_B}\in L_i\right].$$
Using reflection principle several times, we have $p_1$ is the largest one among all the $p_i$'s, then we have the desired results.
\end{proof}

\begin{myprop}\label{prop_hit}
For all $n\ge0$, for all $w\in W_n$. For all $k\ge n$, for all $x,y\in K_{w53}\cap V_k$, for all path $\gamma$ in $V_k$ from $y$ to $\dd K_w\cap V_k$, see Figure \ref{fig_hit} and \cite[FIGURE 2]{BKNPPT12}, we have
$$\bbP_x\left[X^{(k)}\text{ hits }\gamma\text{ before }\tau_{\mathrm{int}(K_w)}\right]\ge\frac{1}{12^{41}}.$$

\begin{figure}[ht]
\centering
\begin{tikzpicture}[scale=1/4]

\draw[very thick] (4*1.73205080756887,4)--(4*1.73205080756887,-4)--(0,-8)--(-4*1.73205080756887,-4)--(-4*1.73205080756887,4)--(0,8)--cycle;

\pgfmathsetmacro{\aa}{4*1.73205080756887};
\pgfmathsetmacro{\ba}{4};

\draw (\aa+2*1.7320508075688772,\ba+2)-- (\aa+2*1.7320508075688772,\ba-2)--(\aa,\ba-4)--(\aa-2*1.7320508075688772,\ba-2)--(\aa-2*1.7320508075688772,\ba+2)--(\aa,\ba+4)--cycle;

\draw (\aa-2*1.7320508075688772,\ba-2)--(\aa+2*1.7320508075688772,\ba+2);
\draw (\aa-2*1.7320508075688772,\ba+2)--(\aa+2*1.7320508075688772,\ba-2);
\draw (\aa,\ba+4)--(\aa,\ba-4);

\draw[very thick] (\aa+2*1.7320508075688772,\ba-2)--(\aa+2*1.7320508075688772,\ba+2)--(\aa,\ba+4);

\pgfmathsetmacro{\aa}{4*1.73205080756887};
\pgfmathsetmacro{\ba}{-4};


\draw (\aa+2*1.7320508075688772,\ba+2)-- (\aa+2*1.7320508075688772,\ba-2)--(\aa,\ba-4)--(\aa-2*1.7320508075688772,\ba-2)--(\aa-2*1.7320508075688772,\ba+2)--(\aa,\ba+4)--cycle;

\draw (\aa-2*1.7320508075688772,\ba-2)--(\aa+2*1.7320508075688772,\ba+2);
\draw (\aa-2*1.7320508075688772,\ba+2)--(\aa+2*1.7320508075688772,\ba-2);
\draw (\aa,\ba+4)--(\aa,\ba-4);

\draw[very thick] (\aa+2*1.7320508075688772,\ba+2)--(\aa+2*1.7320508075688772,\ba-2)--(\aa,\ba-4);

\pgfmathsetmacro{\aa}{0};
\pgfmathsetmacro{\ba}{-8};

\draw (\aa+2*1.7320508075688772,\ba+2)-- (\aa+2*1.7320508075688772,\ba-2)--(\aa,\ba-4)--(\aa-2*1.7320508075688772,\ba-2)--(\aa-2*1.7320508075688772,\ba+2)--(\aa,\ba+4)--cycle;

\draw (\aa-2*1.7320508075688772,\ba-2)--(\aa+2*1.7320508075688772,\ba+2);
\draw (\aa-2*1.7320508075688772,\ba+2)--(\aa+2*1.7320508075688772,\ba-2);
\draw (\aa,\ba+4)--(\aa,\ba-4);

\draw[very thick] (\aa+2*1.7320508075688772,\ba-2)--(\aa,\ba-4)--(\aa-2*1.7320508075688772,\ba-2);

\pgfmathsetmacro{\aa}{-4*1.73205080756887};
\pgfmathsetmacro{\ba}{-4};

\draw (\aa+2*1.7320508075688772,\ba+2)-- (\aa+2*1.7320508075688772,\ba-2)--(\aa,\ba-4)--(\aa-2*1.7320508075688772,\ba-2)--(\aa-2*1.7320508075688772,\ba+2)--(\aa,\ba+4)--cycle;

\draw (\aa-2*1.7320508075688772,\ba-2)--(\aa+2*1.7320508075688772,\ba+2);
\draw (\aa-2*1.7320508075688772,\ba+2)--(\aa+2*1.7320508075688772,\ba-2);
\draw (\aa,\ba+4)--(\aa,\ba-4);

\draw[very thick] (\aa-2*1.7320508075688772,\ba+2)--(\aa-2*1.7320508075688772,\ba-2)--(\aa,\ba-4);

\pgfmathsetmacro{\aa}{-4*1.73205080756887};
\pgfmathsetmacro{\ba}{4};

\draw (\aa+2*1.7320508075688772,\ba+2)-- (\aa+2*1.7320508075688772,\ba-2)--(\aa,\ba-4)--(\aa-2*1.7320508075688772,\ba-2)--(\aa-2*1.7320508075688772,\ba+2)--(\aa,\ba+4)--cycle;

\draw (\aa-2*1.7320508075688772,\ba-2)--(\aa+2*1.7320508075688772,\ba+2);
\draw (\aa-2*1.7320508075688772,\ba+2)--(\aa+2*1.7320508075688772,\ba-2);
\draw (\aa,\ba+4)--(\aa,\ba-4);

\draw[very thick] (\aa-2*1.7320508075688772,\ba-2)--(\aa-2*1.7320508075688772,\ba+2)--(\aa,\ba+4);

\pgfmathsetmacro{\aa}{0};
\pgfmathsetmacro{\ba}{8};

\draw (\aa+2*1.7320508075688772,\ba+2)-- (\aa+2*1.7320508075688772,\ba-2)--(\aa,\ba-4)--(\aa-2*1.7320508075688772,\ba-2)--(\aa-2*1.7320508075688772,\ba+2)--(\aa,\ba+4)--cycle;

\draw (\aa-2*1.7320508075688772,\ba-2)--(\aa+2*1.7320508075688772,\ba+2);
\draw (\aa-2*1.7320508075688772,\ba+2)--(\aa+2*1.7320508075688772,\ba-2);
\draw (\aa,\ba+4)--(\aa,\ba-4);

\draw[very thick] (\aa+2*1.7320508075688772,\ba+2)--(\aa,\ba+4)--(\aa-2*1.7320508075688772,\ba+2);

\draw[fill=black] (4*1.7320508075688772,0) circle (0.2);
\draw[fill=black] (4*1.7320508075688772,-2) circle (0.2);
\draw[fill=black] (4*1.7320508075688772,-4) circle (0.2);

\draw (4*1.7320508075688772-0.7,-1.2) node {{\tiny{$L_2$}}};
\draw (4*1.7320508075688772-0.7,-2.8) node {{\tiny{$L_1$}}};

\draw[densely dotted] (0,12)--(6*1.7320508075688772,6)--(6*1.7320508075688772,-6)--(0,-12)--(-6*1.7320508075688772,-6)--(-6*1.7320508075688772,6)--cycle; 

\draw[densely dotted] (0,4)--(2*1.7320508075688772,2)--(2*1.7320508075688772,-2)--(0,-4)--(-2*1.7320508075688772,-2)--(-2*1.7320508075688772,2)--cycle;

\draw[densely dotted] (5*1.7320508075688772,1)--(5*1.7320508075688772,-1);
\draw[densely dotted] (3*1.7320508075688772,1)--(3*1.7320508075688772,-1);

\draw[densely dotted] (-5*1.7320508075688772,1)--(-5*1.7320508075688772,-1);
\draw[densely dotted] (-3*1.7320508075688772,1)--(-3*1.7320508075688772,-1);

\draw[densely dotted] (1.7320508075688772,-5)--(2*1.7320508075688772,-4);
\draw[densely dotted] (2*1.7320508075688772,-8)--(3*1.7320508075688772,-7);

\draw[densely dotted] (-1.7320508075688772,-5)--(-2*1.7320508075688772,-4);
\draw[densely dotted] (-2*1.7320508075688772,-8)--(-3*1.7320508075688772,-7);

\draw[densely dotted] (1.7320508075688772,5)--(2*1.7320508075688772,4);
\draw[densely dotted] (2*1.7320508075688772,8)--(3*1.7320508075688772,7);

\draw[densely dotted] (-1.7320508075688772,5)--(-2*1.7320508075688772,4);
\draw[densely dotted] (-2*1.7320508075688772,8)--(-3*1.7320508075688772,7);

\end{tikzpicture}
\caption{$X^{(n)}$ hits $\gamma$ before $\tau$}\label{fig_hit}
\end{figure}
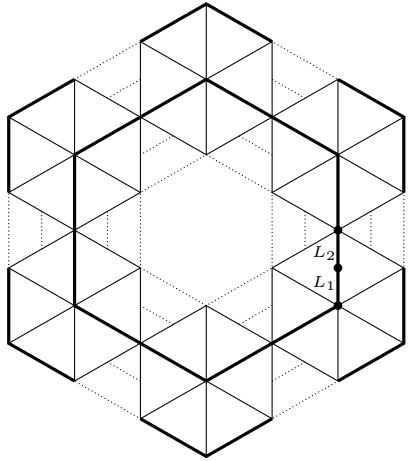
\end{myprop}

\begin{proof}
Starting from $x\in K_{w53}\cap V_k$, $X^{(k)}$ hits the inner thick hexagon in Figure \ref{fig_hit} almost surely. We only need to construct a \emph{closed} curve starting from the inner thick hexagon and surrounding the inner thick hexagon.

By symmetry, we only need to consider the cases $x\in L_1\cap V_k$ and $x\in L_2\cap V_k$. If $x\in L_1\cap V_k$, then using 25 times corner moves, 7 times knight move \Rmnum{1} and 7 times knight move \Rmnum{2}, we obtain a closed curve surrounding the inner thick hexagon. If $x\in L_2\cap V_k$, then using one more time knight move \Rmnum{2} and one more time corner move, we return to the case $x\in L_1\cap V_k$. Therefore, using at most 41 times moves, we obtain a closed curve surrounding the inner thick hexagon.

Combining Lemma \ref{lem_corner}, Lemma \ref{lem_knight1} and Lemma \ref{lem_knight2}, we obtain the desired result.
\end{proof}

\begin{proof}[Proof of Theorem \ref{thm_Harnack}]
By Proposition \ref{prop_cellandball}, we only need to prove the following result.

There exists some positive constant $C$ such that for all $n\ge0$, for all $w\in W_n$, for all $k\ge n$, for all non-negative harmonic function $u$ in $V_k\cap\mathrm{int}(K_w)$, we have
$$\max_{V_k\cap K_{w53}}u\le C\min_{V_k\cap K_{w53}}u.$$

For all subset $A$ of $\dd K_w$, denote
$$h_k(x,A)=\bbP_x\left[X^{(k)}_{\tau_{\mathrm{int}(K_w)}}\in A\right].$$
We only need to show that there exists some universal positive constant $\delta$ such that
$$h_k(x,A)\ge\delta h_k(y,A)\text{ for all }x,y\in V_k\cap K_{w53}.$$

Indeed, let $M_l=h_k(X^{(k)}_{l\wedge\tau_{\mathrm{int}(K_w)}},A)$, then $M_l$ is a martingale.

For all $\eta\in(0,1)$, let
$$T=\inf\myset{l\ge0:M_l<\eta h_k(y,A)}\wedge\tau_{\mathrm{int}(K_w)}.$$
Then
\begin{align*}
h_k(y,A)&=\bbE_yh_k(X^{(k)}_T,A)=\bbE_y\left[h_k(X^{(k)}_T,A)1_{T=\tau_{\mathrm{int}(K_w)}}\right]+\bbE_y\left[h_k(X^{(k)}_T,A)1_{T<\tau_{\mathrm{int}(K_w)}}\right]\\
&\le\bbP_y\left[T=\tau_{\mathrm{int}(K_w)}\right]+\eta h_k(y,A)\bbP_y\left[T<\tau_{\mathrm{int}(K_w)}\right]\\
&=1-\bbP_y\left[T<\tau_{\mathrm{int}(K_w)}\right]+\eta h_k(y,A)\bbP_y\left[T<\tau_{\mathrm{int}(K_w)}\right],
\end{align*}
hence
$$\bbP_y\left[T<\tau_{\mathrm{int}(K_w)}\right]\le\frac{1-h_k(y,A)}{1-\eta h_k(y,A)}<1,$$
hence $\bbP_y\left[T=\tau_{\mathrm{int}(K_w)}\right]>0$, hence there exists some path $\gamma=\myset{\gamma(0),\ldots,\gamma(l_0)}$ from $y$ to $\dd K_w$ such that
$$h_k(\gamma(l),A)\ge\eta h_k(y,A)\text{ for all }l=0,\ldots,l_0.$$
Let
$$S=\inf\myset{l\ge0:X^{(k)}_l\in\gamma},$$
then by Proposition \ref{prop_hit}, we have $\bbP_x\left[S<\tau_{\mathrm{int}(K_w)}\right]\ge12^{-41}$, hence
\begin{align*}
h_k(x,A)&=\bbE_xh_k(X^{(k)}_{S\wedge\tau_{\mathrm{int}(K_w)}},A)\ge\bbE_x\left[h_k(X^{(k)}_{S\wedge\tau_{\mathrm{int}(K_w)}},A)1_{S<\tau_{\mathrm{int}(K_w)}}\right]\\
&\ge\eta h_k(y,A)\bbP_x\left[S<\tau_{\mathrm{int}(K_w)}\right]\ge\frac{1}{12^{41}}\eta h_k(y,A).
\end{align*}
Since $\eta\in(0,1)$ is arbitrary, we have $h_k(x,A)\ge12^{-41}h_k(y,A)$.
\end{proof}

\bibliographystyle{plain}

\def\cprime{$'$}

\end{document}